\documentclass[11pt]{article}
\usepackage[margin=1in]{geometry}
\usepackage{amsmath,amssymb,amsthm}
\usepackage{hyperref}
\usepackage{tikz}

\newtheorem{theorem}{Theorem}
\newtheorem{proposition}{Proposition}
\newtheorem{lemma}{Lemma}
\newtheorem{corollary}{Corollary}
\newtheorem{remark}{Remark}
\newtheorem{question}{Question}

\newcommand{\R}{\mathbb{R}}
\newcommand{\Kcal}{\mathcal{K}}
\newcommand{\Lcal}{\mathcal{L}}

\DeclareMathOperator{\Gr}{Gr}
\DeclareMathOperator{\St}{St}

\title{Discriminant Varieties for Stick Knots and Links}
\author{Alexander Kolpakov \and Igor Rivin}
\date{August 2, 2026}

\hypersetup{
  pdftitle={Discriminant Varieties for Stick Knots and Links},
  pdfauthor={Alexander Kolpakov and Igor Rivin},
  pdfsubject={Upper and lower bounds for the number of knot and link types realizable with a bounded number of sticks},
  pdfkeywords={polygonal knots, stick number, semialgebraic chambers, fewnomials, Garside normal form, Lean 4}
}

\begin{document}

\maketitle

\begin{abstract}
How many knot types can be built from a fixed budget of straight sticks?  We
prove that the answer has factorial-scale growth, settling its order for the
first time.  No previously published general upper bound improves on the
exponential-in-the-square estimate obtained from crossing-number enumeration;
we replace it with a factorial-scale upper bound, which is optimal at the level
of growth order.  The proof turns polygonal self-intersection into a sparse
real-algebraic chamber problem in only linearly many dimensions, while a
complementary braid construction supplies factorially many distinct knots.
The result creates a direct bridge between knot topology, real algebraic
geometry, fewnomial structure, and permutation combinatorics.
\end{abstract}

\section{Introduction}

Let \(p_1,\ldots,p_N\in\R^3\) be labeled vertices joined cyclically by straight
segments.  The space of embedded \(N\)-gons is the complement of a
self-intersection discriminant.  Its connected components are sometimes called
geometric knot types: two polygons in the same component are connected by an
isotopy through \(N\)-gons.  Calvo computed this space for \(N=6,7\), showing in
particular that geometric equivalence is finer than ordinary topological knot
type \cite{Calvo}.  The exact number of components is unknown in general.

The semialgebraic starting point is also already present in Calvo's work:
for every pair of non--incident edges, the possible-intersection locus is
described there as a codimension-one cubic semialgebraic piece, with additional
inequalities selecting intersections of the segments rather than merely their
supporting lines.  Thus neither the polygon-space discriminant nor its
semialgebraicity is claimed as new here.  Our upper bound is a quantitative
extraction from that model: we divide out affine freedom and retain the
binomial dependence on the number of walls in a refined sign-condition
component estimate.

The purpose of this note is to isolate what can be proved by elementary
semialgebraic methods.  The basic facts are:
\begin{itemize}
    \item the discriminant is semialgebraic and has explicit low-degree chart
    equations;
    \item knot type is constant on each chamber of the complement;
    \item every knot type with stick number at most \(N\) occurs in at least one
    chamber of positive measure; and
    \item determinant-related sign-condition bounds, with their binomial
    factor retained after affine normalization, give
    \[
        C_N\le 2(A N)^{3N-12}=N^{3N+o(N)};
    \]
    \item combined with the factorial lower bound implicit in earlier
    arc-index work, the chamber estimate identifies the logarithmic growth
    order
    \(\log|\Kcal_N|=\Theta(N\log N)\), equivalently the scale
    \(N^{\Theta(N)}\);
    \item a further extraction from the Malyutin--Stupakov pure-braid
    construction, using Garside normal forms and doubly alternating
    permutations, sharpens the lower side to
    \(|\Kcal_N|\ge N^{(2/3+o(1))N}\), already among prime satellite knots;
    and
    \item after summing over component decompositions, there are at most
    \(N^{4N+o(N)}\) ordinary links with at most \(N\) total sticks.
\end{itemize}
The determinant form of the crossing walls is the reason that one expects a
substantially structured chamber-count problem rather than an arbitrary
arrangement of high-degree hypersurfaces.

A bound of some kind was already implicit: a generic projection of an
\(N\)-stick polygon has at most \(N(N-3)/2\) crossings, and classical
exponential enumeration by crossing number gives
\(|\Kcal_N|\le\exp(O(N^2))\) \cite{FreedmanHeWang}.  Our contribution is the
much smaller direct stick-space estimate obtained by making Calvo's
semialgebraic discriminant model quantitative:
\[
 |\Kcal_N|\le N^{(3+o(1))N},
\]
which, combined with the factorial arc-index lower bound recalled below,
identifies the order of \(\log|\Kcal_N|\).  The recent strong-geometry reconstruction
theorem of Gros--Ramirez Alfonsin shows that enriched chirotope data determine
the isotopy type of a generic polygonal knot \cite{GrosRamirez}; it does not
give the chamber-count estimate proved here.  The present result is therefore
a quantitative synthesis of established geometric and sign-condition inputs,
not a new discriminant encoding or a new general real-algebraic theorem.
Randell--Simon--Tokle directly studied the counting
function by polygon length and proved an exponential lower bound, but no
global upper bound of this scale \cite{RandellSimonTokle}.  Millett later
listed the corresponding question for equilateral polygons as an open problem
and reported that no useful estimate by edge number was then available
\cite[Problem 5.2]{MillettPhysical}.  After also checking the strong-geometry
reconstruction literature and recent computational work on individual stick
numbers \cite{CantarellaEtAl2025}, we have not located a prior published
statement of either an \(N^{O(N)}\) upper bound for \(|\Kcal_N|\) or the
resulting order \(\log|\Kcal_N|=\Theta(N\log N)\).

\section{Normalized stick-knot parameter space}

Let
\[
    H=\mathbf{1}^{\perp}\subset\R^N.
\]
After subtracting the centroid of the vertices, the three coordinate rows of a
full-dimensional polygon lie in \(H\).  Applying an invertible linear change of
coordinates in \(\R^3\), one may normalize the row Gram matrix to the identity.
This leads to the Stiefel model
\[
    X_N=\St_3(H)
    =
    \{Q\in\R^{3\times N}: Q\mathbf{1}=0,\; QQ^T=I_3\}.
\]
The columns \(q_i=Qe_i\in\R^3\) are the normalized vertices.  Since
\(\dim H=N-1\),
\[
    \dim X_N
    =
    3(N-1)-\frac{3\cdot4}{2}
    =
    3N-9.
\]
Quotienting by \(SO(3)\) gives the oriented Grassmannian
\[
    \widetilde{\Gr}_3(H)=X_N/SO(3),
    \qquad
    \dim \widetilde{\Gr}_3(H)=3N-12.
\]
This does not change chamber counts: the quotient is a locally trivial bundle
with connected fiber \(SO(3)\), and rotations move a polygon inside the same
embedding chamber.  The three-dimensional reduction is useful below.
Quotienting by all of \(O(3)\), on the other hand, identifies mirror images.

\section{Local equations for crossing walls}

Consider two non--incident directed edges \(ab\) and \(cd\), and write
\[
    v=p_b-p_a,\qquad w=p_d-p_c,\qquad r=p_c-p_a.
\]
In the \(xy\)-projection chart, set
\[
    D=\det\begin{pmatrix}
        x_b-x_a & x_d-x_c\\
        y_b-y_a & y_d-y_c
    \end{pmatrix},
\]
\[
    N_\lambda=\det\begin{pmatrix}
        x_c-x_a & x_d-x_c\\
        y_c-y_a & y_d-y_c
    \end{pmatrix},\qquad
    N_\mu=\det\begin{pmatrix}
        x_c-x_a & x_b-x_a\\
        y_c-y_a & y_b-y_a
    \end{pmatrix}.
\]
When \(D\ne0\), the two projected supporting lines meet at parameters
\[
    \lambda=\frac{N_\lambda}{D},\qquad
    \mu=\frac{N_\mu}{D}.
\]
The projected crossing lies in the interiors of both segments precisely when
\[
    N_\lambda(D-N_\lambda)>0,\qquad
    N_\mu(D-N_\mu)>0.
\]
The height difference at the crossing is
\[
    h=z_a+\lambda(z_b-z_a)-z_c-\mu(z_d-z_c).
\]
Clearing denominators gives
\[
    F_{ab,cd}
    =
    D(z_a-z_c)
    +N_\lambda(z_b-z_a)
    -N_\mu(z_d-z_c).
\]
Equivalently, up to sign,
\[
    F_{ab,cd}
    =
    \det(p_b-p_a,\;p_d-p_c,\;p_c-p_a).
\]
This is the affine orientation determinant of the four vertices
\(p_a,p_b,p_c,p_d\), up to sign.  Thus each over/under wall is not merely a
low-degree equation: it is one of the determinant walls governing the order
type of the vertex configuration.  As a polynomial in the coordinate variables,
it has at most \(24\) monomial terms before cancellation.  The same construction
is used in the \(xz\)- and \(yz\)-projection charts.  The cases where the
relevant \(2\times2\) projected determinant vanishes are lower-dimensional
chart degeneracies.  They may be included in a full chartwise sign
description, but are not needed for the enlarged coordinate-free wall family
used in Theorem~\ref{thm:upper}.

\section{The discriminant and its chambers}

Let \(\Delta_N\subset X_N\) be the set for which the cyclic piecewise-linear
map is not an embedding.  Thus an edge may collapse, adjacent edges may meet
away from their common endpoint, or two combinatorially non--incident edges may
meet.  The last condition says that, for some such pair \(ab,cd\), there are
\(\lambda,\mu\in[0,1]\) with
\[
    p_a+\lambda(p_b-p_a)=p_c+\mu(p_d-p_c).
\]
Edge collapse and adjacent-edge overlap have similarly elementary polynomial
descriptions.  These first-order formulas use only polynomial equalities and
inequalities with existentially quantified segment parameters, so the
Tarski--Seidenberg theorem \cite[Section 2.2]{BochnakCosteRoy} implies that
\(\Delta_N\) is semialgebraic.  This formulation includes the adjacent
backtracking degeneracy that is missed if one mentions only non--incident edge
pairs.

Define
\[
    C_N=b_0(X_N\setminus\Delta_N),
\]
the number of embedding chambers.  For a knot type \(K\), let \(C_N(K)\) be
the number of chambers in which the polygon has type \(K\).

\begin{proposition}[Chambers carry one knot type]\label{prop:chambers}
The ambient isotopy class of the cyclic polygon is constant on each connected
component of \(X_N\setminus\Delta_N\).
\end{proposition}

\begin{proof}
Connected semialgebraic sets are semialgebraically path-connected, so two
points of a chamber can be joined by a path in
\(X_N\setminus\Delta_N\).  Such a path is a one-parameter family of embedded
polygonal curves with fixed cyclic combinatorics.  For a sufficiently fine
subdivision of the parameter interval, consecutive polygons are close enough
that the straight-line interpolation of corresponding vertices remains
embedded.  This gives an isotopy through embedded polygonal curves, and the
ambient isotopy extension theorem \cite{EdwardsKirby} gives ambient isotopy of
the endpoints.
\end{proof}

\begin{remark}[Generic walls]
At a smooth codimension-one point of a single wall \(F_{ab,cd}=0\), with
\(D\ne0\), \(0<\lambda,\mu<1\), and no simultaneous degeneracy, crossing the
wall changes the sign of the height difference \(h=F_{ab,cd}/D\).  Thus the
two adjacent diagrams differ by one over/under switch.  This is a local
statement; it does not imply that all crossing assignments in a large diagram
are independently realizable.
\end{remark}

\section{Support of the knot type map}

Let \(\operatorname{stick}(K)\) denote the stick number of a knot type \(K\).
This is the minimum number of straight segments needed in a polygonal
representative of \(K\); see, for example, the standard tables and references in
KnotInfo \cite{KnotInfoStick} and the stick-number work of Adams et al.
\cite{AdamsStick}.

\begin{theorem}[Support theorem]\label{thm:support}
For \(N\ge4\), the image of the knot type map
\[
    \tau_N:X_N\setminus\Delta_N\longrightarrow
    \{\text{ambient isotopy classes of knots}\}
\]
is exactly
\[
    \Kcal_N=\{K:\operatorname{stick}(K)\le N\}.
\]
Moreover, if \(K\in\Kcal_N\), then \(\tau_N^{-1}(K)\) contains a nonempty open
subset of \(X_N\).  In particular it has positive measure for the natural
\(O(N-1)\)-invariant Riemannian volume on \(X_N\), and
\[
    C_N(K)\ge1
    \quad\Longleftrightarrow\quad
    \operatorname{stick}(K)\le N.
\]
\end{theorem}

\begin{proof}
If \(Q\in X_N\setminus\Delta_N\), then its columns form an embedded polygon
with \(N\) sticks.  Hence the represented knot type \(K=\tau_N(Q)\) satisfies
\(\operatorname{stick}(K)\le N\).

Conversely, suppose \(\operatorname{stick}(K)\le N\).  Choose a polygonal
representative of \(K\) with at most \(N\) sticks and subdivide edges, if
necessary, to obtain exactly \(N\) labeled vertices.  After an arbitrarily small
ambient isotopy, assume that the centered coordinate rows span a
three-dimensional subspace of \(H\).  Let \(M\in\R^{3\times N}\) be the matrix
of centered vertex coordinates.  Then \(M\mathbf 1=0\) and \(MM^T\) is positive
definite.  Put
\[
    A=(MM^T)^{1/2},\qquad Q=A^{-1}M.
\]
Then \(Q\mathbf 1=0\) and \(QQ^T=I_3\), so \(Q\in X_N\).  The original centered
polygon is the image of the \(Q\)-polygon under the orientation-preserving
linear map \(A\).  Therefore \(Q\) represents the same knot type \(K\), and
\(Q\notin\Delta_N\).

Finally, stability of PL embeddings under sufficiently small vertex
perturbations \cite{EdwardsKirby} gives a neighborhood of \(Q\) in \(X_N\)
representing the same knot type.  Concretely, non--incident edges have
positive separation, every edge is nonzero, and local injectivity at the
finitely many adjacent pairs persists; at a straight subdivision vertex the
two outgoing directions have strictly negative dot product.  Nonempty open
subsets of the compact smooth manifold \(X_N\) have positive natural
\(O(N-1)\)-invariant Riemannian volume.
\end{proof}

\section{A determinantal semialgebraic upper bound}

\begin{theorem}[Chamber upper bound]\label{thm:upper}
For \(N\ge4\), there is an absolute constant \(A>0\) such that
\[
    C_N\le 2(A N)^{3N-12}=N^{3N+o(N)}.
\]
For the complete-chord model, in which every pair of labeled vertices is
allowed as a chord, the analogous bound is
\[
    C_N^{\mathrm{all}}\le 2(A N^3)^{3N-12}=N^{9N+o(N)}.
\]
\end{theorem}

\begin{proof}
Write \(e_i=p_{i+1}-p_i\), with indices read cyclically.  There are
\[
    r_N=\frac{N(N-3)}{2}
\]
unordered pairs of non--incident edges.  For each such pair put
\[
    f_{ij}=\det(e_i,e_j,p_j-p_i),
\]
and, for each adjacent pair, put
\[
    h_i=\|e_i\times e_{i+1}\|^2.
\]
The \(f_{ij}\) have degree three and the \(h_i\) have degree four.  If none
vanishes, consecutive edges are nonzero and noncollinear, hence meet only at
their prescribed common endpoint.  Two non--incident segments cannot meet
either: an intersection would make their two directions and displacement
coplanar and force the corresponding \(f_{ij}\) to vanish.  Thus the complement
of these
\[
    s=r_N+N=\frac{N(N-1)}2
\]
algebraic walls consists entirely of embedded polygons.  It is harmless that
these enlarged walls also delete benign coplanar configurations and straight
subdivisions.

We next remove the affine degrees of freedom before counting.  For
\(\varepsilon\in\{+1,-1\}\), let \(S_\varepsilon\) be the affine slice
\[
 p_1=0,\qquad p_2=(1,0,0),\qquad p_3=(0,1,0),\qquad
 p_4=(0,0,\varepsilon).
\]
It is an affine space of dimension
\[
    m=3(N-4)=3N-12.
\]
Every embedding chamber in \(X_N\) contains a polygon whose first four
vertices are affinely independent.  Indeed, perturb the vertices arbitrarily
slightly in the centered full-row-rank coordinate space so that the first four
are affinely independent; embedding and full rank persist, and subsequent
row-normalization returns a nearby point of the same chamber.  According to the sign of that
determinant, a unique orientation-preserving affine map carries the four
vertices to the fixed quartet in \(S_+\) or \(S_-\).  The group of
orientation-preserving affine maps is connected.  Following such a path and
continuously centering and row-normalizing the polygon stays in the original
embedding chamber.  Conversely, centering and row-normalizing gives a
continuous map from the embedded locus \(E_\varepsilon\subset S_\varepsilon\)
to \(X_N\setminus\Delta_N\).  Consequently the components of the two slice
loci map onto the embedding chambers, and
\[
    C_N\le b_0(E_+)+b_0(E_-).
\]

Restrict the \(s\) wall polynomials to either slice, discard nonzero constants,
and let \(G_\varepsilon\) be the complement of their zero sets.  Every remaining
restriction is a proper polynomial: for a determinant wall one may vary a
free endpoint off the plane of the other three, and for an adjacent wall off
the relevant line; walls involving only the fixed tetrahedron are nonzero
constants.  Hence their finite union has empty interior in the slice.  Every
open component of \(E_\varepsilon\) therefore meets
\(G_\varepsilon\), while \(G_\varepsilon\subset E_\varepsilon\).  The induced
map on components is surjective, so
\[
    b_0(E_\varepsilon)\le b_0(G_\varepsilon).
\]

On \(G_\varepsilon\) the sign of every wall is locally constant, so
\(b_0(G_\varepsilon)\) is the sum of the component counts of its realizable
strict sign conditions.  Dimension-sensitive cell bounds retaining the
essential \((s/m)^m\) dependence go back to Pollack--Roy
\cite{PollackRoy}.  We use the part of the sign-condition theorem that is lost in the cruder
\(s^m\) notation.  The explicit Barone--Basu bound
\cite[Theorem 1.1]{BaroneBasu}, specialized to \(s\) polynomials of degree at
most four on \(\R^m\), implies
\[
 b_0(G_\varepsilon)
 \le B^m\sum_{j=0}^{m}\binom{s+1}{j}
\]
for an absolute constant \(B\).  Explicitly, take
\(\mathcal Q=\varnothing\), \(V=\R^m\), \(k'=k=m\), and \(d=4\) in their
formula (with any fixed vacuous \(d_0\), say \(2\)).  Its \(j\)-th auxiliary
factor is bounded using
\(\binom{m+1}{j+1}\le2^{m+1}\), while \(4^j\) and all fixed-degree factors
are absorbed into \(B^m\).  For \(N\ge5\) one has
\(m\le(s+1)/2\), and hence
\[
 \sum_{j=0}^{m}\binom{s+1}{j}
 \le (m+1)\binom{s+1}{m}
 \le (m+1)\left(\frac{e(s+1)}m\right)^m.
\]
Since \((s+1)/m=O(N)\), enlarging the absolute constant absorbs the factor
\(m+1\) and gives \(b_0(G_\varepsilon)\le(AN)^m\).  The case \(N=4\) has two
orientation components and is covered by the prefactor \(2\).  This proves
the first estimate.

For the complete straight-line graph on the \(N\) vertices there are
\(3\binom N4\) unordered pairs of vertex-disjoint chords.  In addition one
must exclude overlap or backtracking for the
\(N\binom{N-1}{2}=O(N^3)\) pairs of distinct chords incident at a common
vertex, using the analogous squared-cross-product walls.  Thus the enlarged
wall family still has \(s=O(N^4)\), so \(s/m=O(N^3)\), and the identical
binomial-sum argument gives \(2(AN^3)^{3N-12}\).
\end{proof}

\begin{corollary}[Knot types with bounded stick number]\label{cor:knot-types}
Let \(\Kcal_N=\{K:\operatorname{stick}(K)\le N\}\).  Then
\[
    |\Kcal_N|
    \le
    \sum_{K\in\Kcal_N} C_N(K)
    =
    C_N
    \le
    2(A N)^{3N-12}.
\]
\end{corollary}

\begin{proof}
By Theorem~\ref{thm:support}, every \(K\in\Kcal_N\) is represented by at least
one chamber.  By Proposition~\ref{prop:chambers}, each chamber represents only
one topological knot type.  Therefore the chamber count bounds the number of
represented knot types from above.  The final inequality is Theorem~\ref{thm:upper}.
\end{proof}

\begin{theorem}[Asymptotic growth order]\label{thm:asymptotic-order}
As \(N\to\infty\),
\[
    \log|\Kcal_N|=\Theta(N\log N),
\]
or equivalently \(|\Kcal_N|=N^{\Theta(N)}\).
\end{theorem}

\begin{proof}
The upper bound follows from Corollary~\ref{cor:knot-types}.  For the lower
bound, Malyutin--Stupakov prove that the number of oriented prime-knot classes
with arc index at most \(k\) is at least
\[
    \left\lfloor\frac{k+4}{11}\right\rfloor!
\]
for \(k>5\) \cite[Theorem 1]{MalyutinStupakovArc}.  Take
\(k=\lfloor2N/3\rfloor+1\).  The Huh--Oh inequality
\cite[Theorem 3]{HuhOh} gives
\[
    \operatorname{stick}(K)
    \le \frac32(k-1)
    =\frac32\left\lfloor\frac{2N}{3}\right\rfloor
    \le N.
\]
Thus every resulting knot is within the stick budget, and forgetting
orientation costs at most a factor of two.  Hence
\[
 |\Kcal_N|
 \ge \frac12
 \left\lfloor\frac{\lfloor2N/3\rfloor+5}{11}\right\rfloor!
 =N^{(2/33+o(1))N}.
\]
Stirling's formula and the upper estimate
\(|\Kcal_N|\le N^{3N+o(N)}\) complete the proof.
\end{proof}

\begin{remark}[Upper bounds for links]\label{rem:link-upper}
Let \(\Lcal_N\) be the set of ambient isotopy classes of ordinary unoriented
unordered links in \(\R^3\) that admit a polygonal representative with at most
\(N\) total sticks.  The same semialgebraic upper-bound mechanism gives, after
summing over component decompositions, an estimate
\[
    |\Lcal_N|\le N^{4N+o(N)}.
\]

Indeed, a fixed labeled link combinatorics is a cyclic decomposition of the
\(N\) labeled vertices, equivalently a permutation of those vertices up to the
irrelevant choices of starting point on each cycle.  Thus there are at most
\(N!\) labeled component decompositions to consider.  For each such
decomposition, the number of pairs of non--incident edges is at most \(N^2/2\),
and there are exactly \(N\) cyclically adjacent edge pairs across all
components.  Add the corresponding \(N\) squared-cross-product walls to the
non--incident determinant walls.  The same argument as in
Theorem~\ref{thm:upper} then gives at
most
\[
    2(A N)^{3N-12}
\]
semialgebraic chambers for that fixed decomposition.  Every unmarked unordered
link type with total stick number at most \(N\) is represented in at least one
of these labeled models after subdividing edges and making an arbitrarily small
full-dimensional perturbation.  Passing from labeled component data to ordinary
unordered links can only identify types, not create new ones.  Therefore
\[
    |\Lcal_N|
    \le
    2N!\,(A N)^{3N-12}
    =
    N^{4N+o(N)}.
\]
The link case is not literally the single cyclic-polygon Stiefel model: one
must also account for the number of components and their cyclic decompositions.
\end{remark}

\begin{remark}[Upper bounds versus lower bounds]
The chamber-to-knot-type map is surjective onto \(\Kcal_N\), but it is not
injective.  Many chambers can carry the same isotopy class; this already occurs
in Calvo's low-\(N\) computations \cite{Calvo}.  Consequently an upper bound for
\(C_N\) is automatically an upper bound for \(|\Kcal_N|\), but a lower bound for
\(C_N\) is not a lower bound for the number of knot types unless one separately
distinguishes the isotopy classes carried by those chambers.
\end{remark}

\section{A quantitative lower-bound refinement}

Theorem~\ref{thm:asymptotic-order} already settles the logarithmic growth
order using the printed Malyutin--Stupakov theorem.  We now use their stronger
Proposition~1, which retains the number of permutation-braid factors as a
parameter, to sharpen the lower bound.  Combining that flexibility with a
large language of Garside normal forms gives the estimate proved below.  We
include all combinatorial details because our argument extracts a consequence
of their proposition rather than restating their printed factorial theorem.

For \(w=w_1\cdots w_n\in S_n\), write
\[
    \operatorname{Des}(w)=\{i\in\{1,\ldots,n-1\}:w_i>w_{i+1}\},
    \qquad
    D_n=\{1,3,5,\ldots\}\cap\{1,\ldots,n-1\}.
\]
Let \(E_n\) be the Euler zigzag number, so that
\(E_n=\#\{w:\operatorname{Des}(w)=D_n\}\), and set
\[
    G_n=\{w\in S_n:
       \operatorname{Des}(w)=\operatorname{Des}(w^{-1})=D_n\},
    \qquad g_n=|G_n|.
\]
Thus \(G_n\) is the set of doubly alternating permutations, in Stanley's
down--up convention.

\begin{lemma}[RSK lower bound]\label{lem:doubly-alternating}
For every \(n\ge1\),
\[
    g_n\ge \frac{E_n^2}{n!}.
\]
Moreover,
\[
    \log g_n=\log(n!)-O(n).
\]
\end{lemma}

\begin{proof}
For a partition \(\lambda\vdash n\), let \(f^\lambda\) be the number of
standard Young tableaux of shape \(\lambda\), and let \(a_\lambda\) be the
number of those tableaux whose descent set is \(D_n\).  Under row-insertion
RSK, \(w\mapsto(P(w),Q(w))\), the adjacent-letter bumping lemma gives
\[
    \operatorname{Des}(w)=\operatorname{Des}(Q(w)).
\]
The inversion symmetry
\(\operatorname{RSK}(w^{-1})=(Q(w),P(w))\) therefore also gives
\(\operatorname{Des}(w^{-1})=\operatorname{Des}(P(w))\).  Consequently,
shape by shape,
\[
    E_n=\sum_{\lambda\vdash n}a_\lambda f^\lambda,
    \qquad
    g_n=\sum_{\lambda\vdash n}a_\lambda^2,
    \qquad
    n!=\sum_{\lambda\vdash n}(f^\lambda)^2.
\]
The first identity is also the RSK interpretation recorded by Stanley
\cite[Section 5]{StanleyAlternating}.  Cauchy--Schwarz applied to the first
sum proves \(E_n^2\le g_n n!\).

Finally, Stanley's Euler-number estimate
\cite[Equation (1.10)]{StanleyAlternating} is
\[
    \frac{E_n}{n!}
      =\frac4\pi\left(\frac2\pi\right)^n
       +O\!\left(\left(\frac{2}{3\pi}\right)^n\right).
\]
Hence \(\log E_n=\log(n!)+n\log(2/\pi)+O(1)\).  The preceding lower
bound and the trivial inequality \(g_n\le n!\) now give
\(\log g_n=\log(n!)-O(n)\).
\end{proof}

\begin{lemma}[Many short pure braids]\label{lem:short-pure-braids}
For integers \(n\ge3\) and \(q\ge1\), the pure braid group \(PB_n\) contains
at least
\[
    \left\lceil\frac{g_n^q}{n!}\right\rceil
\]
distinct elements, each expressible as a product of at most \(q+1\) positive
permutation braids.
\end{lemma}

\begin{proof}
Let \(A_w\) be the positive permutation braid inducing \(w\in S_n\).
Elrifai--Morton identify these braids with the simple elements in the classical
Garside structure.  In their convention, the starting and finishing sets are
\[
    S(A_w)=\operatorname{Des}(w),
    \qquad
    F(A_w)=\operatorname{Des}(w^{-1});
\]
the second identity follows from \(F(A)=S(\operatorname{rev}A)\).  Their left
normal-form criterion says that a sequence of simple factors is canonical
when \(S(A_{i+1})\subseteq F(A_i)\), and that the resulting expression is
unique
\cite[Lemmas 2.3 and 2.4 and Theorem 2.9]{ElrifaiMorton}.

For every \(w\in G_n\), both sets above equal \(D_n\).  Since \(n\ge3\),
\(D_n\) is nonempty and proper, so \(A_w\) is neither the identity nor the
Garside half twist.  It follows that every word
\[
    A_{w_1}A_{w_2}\cdots A_{w_q},
    \qquad (w_1,\ldots,w_q)\in G_n^q,
\]
is already its unique left normal form.  The \(g_n^q\) tuples therefore give
distinct positive braids.

Let \(\rho:B_n\to S_n\) be the endpoint-permutation homomorphism.  By the
pigeonhole principle, at least \(\lceil g_n^q/n!\rceil\) of the preceding
braids have a common endpoint permutation \(u\).  Append the same positive
permutation braid \(A_{u^{-1}}\) to each of them.  Every resulting braid is
pure, and right cancellation in \(B_n\) shows that they remain distinct.  If
the appended braid is the identity it may be omitted; hence each result is a
product of at most \(q+1\) positive permutation braids.
\end{proof}

\begin{theorem}[Finite stick-number lower bound]\label{thm:finite-stick-lower}
For \(n\ge3\) and \(q\ge1\), put
\[
    a_{n,q}=n\max\{11,q+9\}-4,
    \qquad
    s_{n,q}=\left\lfloor\frac32\bigl(a_{n,q}-1\bigr)\right\rfloor.
\]
The number of ordinary unoriented prime satellite-knot types with stick number
at most \(s_{n,q}\) is at least
\[
    \left\lceil
       \frac12\left\lceil\frac{g_n^q}{n!}\right\rceil
    \right\rceil
    \ \ge\
    \frac{E_n^{2q}}{2(n!)^{q+1}}.
\]
\end{theorem}

\begin{proof}
Malyutin--Stupakov construct, for every \(n>2\), an injection from \(PB_n\)
to oriented knot types.  Every knot in its image is prime and satellite, and
if the input braid is a product of \(s\) permutation braids, then the image has
arc index at most
\[
    n\max\{11,s+8\}-4
\]
\cite[Proposition 1, pp.~3--4]{MalyutinStupakovArc}.  Apply this injection to
the braids from Lemma~\ref{lem:short-pure-braids}.  Since their factor length
is at most \(q+1\), their images have arc index at most \(a_{n,q}\).

Forgetting knot orientation has fibers of size at most two, so at least the
first displayed number of ordinary knot types remain.  Prime satellite knots
are nontrivial, and Huh--Oh prove for every nontrivial knot that
\[
    \operatorname{stick}(K)
       \le\frac32\bigl(\alpha(K)-1\bigr)
\]
\cite[Theorem 3]{HuhOh}; thus all these knots have stick number at most
\(s_{n,q}\).  The last inequality follows from
Lemma~\ref{lem:doubly-alternating}.
\end{proof}

We now optimize the stick budget as \(N\to\infty\).  For sufficiently large \(N\),
set
\[
    q_N=\lfloor\log N\rfloor,
    \qquad
    n_N=\left\lfloor
       \frac{2N/3+5}{q_N+9}
    \right\rfloor.
\]
Then \(q_N\ge2\), \(n_N\ge3\), and
\[
 \frac32\bigl(n_N(q_N+9)-5\bigr)\le N.
\]
Theorem~\ref{thm:finite-stick-lower} and Lemma~\ref{lem:doubly-alternating}
therefore give prime satellite knots in \(\Kcal_N\) satisfying
\[
 \begin{aligned}
 \log|\Kcal_N|
 &\ge
   2q_N\log E_{n_N}-(q_N+1)\log(n_N!)-\log2 \\
 &= (q_N-1)\log(n_N!)-O(q_Nn_N) \\
 &= (q_N-1)n_N\log n_N-O(q_Nn_N).
 \end{aligned}
\]
Here
\[
 q_Nn_N=O(N),\qquad
 (q_N-1)n_N=\frac23N+O\!\left(\frac{N}{\log N}\right),
 \qquad
 \log n_N=\log N-\log\log N+O(1).
\]
Consequently
\[
    \log|\Kcal_N|
       \ge \frac23N\log N-\frac23N\log\log N-O(N),
\]
and in particular
\[
    |\Kcal_N|\ge N^{(2/3+o(1))N}.
\]
This lower bound already counts only prime satellite knot types.

\begin{remark}[At most \(N\) versus an \(N\)-edge representative]\label{rem:exact-N}
These two formulations count the same knot types.  Starting from any
embedded polygon with fewer than \(N\) edges, subdivide an edge and replace
the new collinear vertex by a sufficiently small local kink.  PL stability
preserves the ambient isotopy class and makes the two new incident edges
genuine noncollinear sticks.  Repeating gives an embedded representative
with exactly \(N\) polygon edges.  Hence the lower and upper bounds in this section
apply both to stick number at most \(N\) and to realizability by an
\(N\)-edge polygon.  This is not a statement about the shell
\(\{K:\operatorname{stick}(K)=N\}\); no comparable lower bound for every
exact-minimal-stick shell is asserted here.
\end{remark}

\begin{corollary}[Refined asymptotic bounds]
As \(N\to\infty\),
\[
 N^{(2/3+o(1))N}
 \le |\Kcal_N|
 \le N^{3N+o(N)}.
\]
In particular
\[
 \log|\Kcal_N|=\Theta(N\log N),
\]
or equivalently \(|\Kcal_N|=N^{\Theta(N)}\).  These asymptotic bounds determine
the order of growth of the logarithm.  They do not furnish an asymptotic
equivalent for \(|\Kcal_N|\).  Nor do they determine whether the normalized
logarithm
\[
    \frac{\log|\Kcal_N|}{N\log N}
\]
has a limit.
\end{corollary}

For ordinary links the lower-bound problem is more direct.  In this paragraph
all links are unoriented and unordered, and
\[
    \Lcal_N
    =
    \{L:\text{\(L\) admits a polygonal representative with at most \(N\)
    total sticks}\}.
\]
The invariant used below is the matrix of absolute pairwise linking numbers,
viewed up to simultaneous permutation of the components.  This is an invariant
of unoriented unordered ambient isotopy.

\begin{proposition}[Addressed linking matrices]\label{prop:addressed-linking}
There is an absolute constant \(B_0>0\) with the following property.  For every
\(m\ge2\) and every symmetric matrix
\[
    A=(a_{ij})_{1\le i,j\le m},\qquad
    a_{ii}=0,\quad a_{ij}\in\mathbb Z_{\ge0},
\]
there is an ordinary unoriented unordered link \(L(A)\) such that
\[
    \operatorname{stick}(L(A))
    \le
    B_0\left(m+\sum_{i<j}a_{ij}\right),
\]
and \(L(A)\) and \(L(A')\) are ambient isotopic only if \(A=A'\).
\end{proposition}

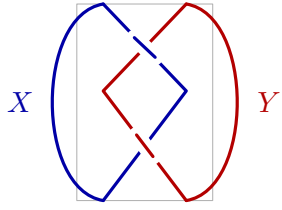
\begin{figure}[ht]
\centering
\begin{tikzpicture}[
    x=1cm,y=1cm,
    line cap=round,line join=round,
    strandX/.style={very thick,blue!65!black},
    strandY/.style={very thick,red!70!black}
]
    \draw[gray!60] (-0.9,-1.3) rectangle (0.9,1.3);

    \draw[strandX]
        (-0.55,1.3)
        .. controls (-1.45,1.15) and (-1.45,-1.15) ..
        (-0.55,-1.3);
    \draw[strandY]
        (0.55,1.3)
        .. controls (1.45,1.15) and (1.45,-1.15) ..
        (0.55,-1.3);

    \draw[strandX]
        (-0.55,1.3) -- (0.55,0.15) -- (-0.55,-1.3);
    \draw[strandY]
        (0.55,1.3) -- (-0.55,0.15) -- (0.55,-1.3);

    \draw[white,line width=6pt] (-0.138,0.855) -- (0.138,0.595);
    \draw[strandX]               (-0.138,0.855) -- (0.138,0.595);

    \draw[white,line width=6pt] (-0.109,-0.432) -- (0.109,-0.718);
    \draw[strandY]              (-0.109,-0.432) -- (0.109,-0.718);

    \node[blue!65!black] at (-1.65,0) {\(X\)};
    \node[red!70!black] at (1.65,0) {\(Y\)};
\end{tikzpicture}
\caption{The componentwise closure of this tangle is a Hopf link.}
\label{fig:hopf-clasp}
\end{figure}

\begin{proof}
We first prescribe an abstract absolute linking matrix and then realize it by a
polygonal link.  Use frame components
\[
    R,F_1,\ldots,F_m,U,V
\]
and payload components
\[
    P_1,\ldots,P_m.
\]
The odd part of the absolute linking matrix is fixed as follows:
\[
    |\operatorname{lk}(R,U)|
    =
    |\operatorname{lk}(R,V)|
    =
    |\operatorname{lk}(R,F_1)|
    =
    1,
\]
\[
    |\operatorname{lk}(F_i,F_{i+1})|=1\quad(1\le i<m),
    \qquad
    |\operatorname{lk}(F_i,P_i)|=1\quad(1\le i\le m),
\]
and all other entries involving a frame component have even value \(0\).  Among
payload components set
\[
    |\operatorname{lk}(P_i,P_j)|=2a_{ij}.
\]

Reducing the absolute linking matrix modulo \(2\) gives a graph.  Its vertices
are the components; its edges are precisely the odd linking pairs above.  This
graph has only one nontrivial automorphism, namely the transposition of the two
leaves \(U,V\).  In particular it fixes \(R\), then fixes the path
\[
    F_1,F_2,\ldots,F_m
\]
in order, and therefore fixes each payload component \(P_i\).  Hence an
unoriented unordered ambient isotopy between two such links must identify the
payload submatrix entry-by-entry.  Since the payload entries are \(2a_{ij}\),
the matrix \(A\) is recovered from the ordinary unmarked link type.

It remains to realize the prescribed matrix with the stated stick cost.  Here
is a concrete routing model.  Give every component a private straight rail
through a long rectangular box, with the rails mutually parallel and
separated.  Allocate disjoint thin slabs along the rail direction, one slab
for each required unit of linking.  In a slab assigned to components \(X,Y\),
leave every other rail straight and replace short subsegments of the \(X\)-
and \(Y\)-rails by a fixed polygonal realization of the two-strand full twist
\(\sigma_1^2\), or its mirror \(\sigma_1^{-2}\), as in
Figure~\ref{fig:hopf-clasp}.  Its endpoints are matched to the original rail
segments on the same respective sides of the slab.  The slabs and small
tubular neighborhoods of all rails are disjoint.  Outside the long box, close
each rail by a return arc in its own separated height plane; these closing arcs
can be chosen as polygonal rectangles and introduce no linking.

After choosing auxiliary orientations, choose between the two mirror-image
full twists so that its two crossings have positive sign.  The crossing
formula
\[
    \operatorname{lk}(X,Y)
    =\frac12\sum_{p\in X\cap Y}\operatorname{sign}(p)
\]
shows that each inserted twist raises \(\operatorname{lk}(X,Y)\) by one.
Because its supporting ball meets no other component, it changes no other
pairwise linking number.  Each component uses a fixed number of sticks for its
rail and return arc, and each clasp replaces two straight subsegments by a
fixed bounded number of sticks.  Thus a diagram
with \(q\) clasps has at most \(c_0m+c_1q\) sticks for universal constants
\(c_0,c_1\).  Since the final invariant uses absolute values, the auxiliary
orientations are irrelevant.  The fixed odd address graph uses \(O(m)\)
clasps, while the payload matrix uses
\(2\sum_{i<j}a_{ij}\) clasps.  Enlarging one universal constant \(B_0\) gives
the claimed estimate.
\end{proof}

\begin{theorem}[Unmarked unordered link production]\label{thm:link-lower}
There is an absolute constant \(B>0\) such that, for all \(n\ge2\),
\[
    |\Lcal_{Bn}|
    \ge
    \binom{\binom n2+n}{n}
    \ge
    n^n\exp(-O(n)).
\]
Consequently, for some constants \(c,C>0\),
\[
    N^{cN}\le |\Lcal_N|\le N^{CN}
\]
for all sufficiently large \(N\).
\end{theorem}

\begin{proof}
Take \(m=n\) in Proposition~\ref{prop:addressed-linking}, and consider all
nonnegative symmetric payload matrices with
\[
    \sum_{i<j}a_{ij}\le n.
\]
There are \(M=\binom n2\) independent entries, so the number of such matrices is
the stars-and-bars number
\[
    \binom{M+n}{n}.
\]
Each produces a distinct unmarked unordered link type, and each has total stick
number at most \(B n\) for a universal \(B\).  Finally,
\[
    \binom{M+n}{n}
    =
    \prod_{k=1}^n \frac{M+k}{k}
    \ge
    \frac{M^n}{n!}
    =
    n^n\exp(-O(n)),
\]
because \(M=\binom n2\).  The upper bound in the final display is
Remark~\ref{rem:link-upper}; the lower bound follows by replacing \(n\) by
\(\lfloor N/B\rfloor\).
\end{proof}

Thus ordinary unmarked unordered links already have the \(N^{\Theta(N)}\)
growth scale, with a completely elementary invariant: the absolute pairwise
linking matrix, together with a parity address frame that is itself part of the
unmarked link.

\section{Machine-checked core and proof boundary}

The accompanying Lean~4 development in \texttt{formal/} checks the algebraic
and finite-combinatorial core of the knot bounds.  On the upper side it verifies
the determinant and cross-product wall implications, the arithmetic identities
for the declared cyclic wall counts and slice-size formula \(3N-12\), the
lower-half range needed in the Barone--Basu sum, and the factorial-sensitive estimate
\[
 \sum_{j=0}^{m}4^j\binom{s}{j}
 \le (m+1)\left(\frac{12s}{m}\right)^m.
\]
It specializes this estimate to a linear base in \(N\) and assembles the final
upper inequality from explicitly named chamber-reduction and component-bound
hypotheses.

For the stronger lower bound, Lean verifies the finite Cauchy--Schwarz step
\(E_n^2\le n!g_n\) once the three RSK counting identities are supplied, the
floor formula
\[
 n=\left\lfloor\frac{2N+15}{3(q+9)}\right\rfloor,
\]
the exact stick-budget slack, and the algebra that propagates supplied count
inequalities already containing the \(n!\) endpoint-fiber loss and factor two
for forgetting orientation.

\paragraph{Exact verification status.}
Neither the chamber upper bound nor the long Garside lower-bound argument is
formalized end-to-end.  On the upper side, PL ambient-isotopy extension, the
topology of polygon chambers, the affine-slice chamber surjection, and the
Barone--Basu component theorem remain cited external mathematics.  The Lean
upper wrapper takes their required consequences as named hypotheses and checks
the determinant, binomial, and final inequality arithmetic after those inputs.
On the lower side, the RSK identities and Euler-number asymptotic, Garside
normal-form uniqueness and the endpoint-fiber count, the
Malyutin--Stupakov braid-to-knot injection, the Huh--Oh stick/arc theorem, and
the final choice \(q=\lfloor\log N\rfloor\) with its asymptotic analysis are
not formalized.  Lean checks the Cauchy--Schwarz consequence, floor and budget
inequalities, and the conditional finite transfer once the finite RSK,
Garside, and knot-theoretic consequences are supplied as hypotheses.  The
logarithmic choice of \(q\) and the asymptotic passage remain wholly outside
Lean.  Consequently, the headline asymptotic theorem is not presented as an
unconditional Lean theorem.

Concretely, the development does not define knot types, ambient isotopy,
semialgebraic chambers, braids, tableaux, or the RSK correspondence.  Its final
wrappers quantify over numerical component and knot counts and assume the
external geometric and knot-theoretic transfers relating those counts.

These theorem parameters are explicit proof boundaries, not project-specific
axioms.  For that reason they do not appear in \texttt{\#print axioms}: the
axiom report checks the foundations of the conditional Lean derivations, not
the external inputs themselves.  There are no \texttt{sorry}, \texttt{admit},
or project-specific \texttt{axiom} declarations; the selected declarations use
only the standard Lean/mathlib foundations printed by the root module.  Running
\texttt{lake build} in \texttt{formal/} builds the complete checked target.

\section{Fewnomial and determinant directions}

The crossing equations have two special features.  First, each wall is a
constant-support polynomial: the affine four-point determinant
\(\det(p_b-p_a,p_d-p_c,p_c-p_a)\) has at most \(24\) monomial terms.  Second,
the family of all such walls is highly correlated: these are determinant walls
coming from one point configuration, not unrelated sparse cubics.  This places
the problem between fewnomial real algebraic geometry and the combinatorics of
order types.

One possible route is determinant-combinatorial.  The signs of all affine
four-point determinants are the chirotope data of the labeled configuration in
\(\R^3\).  Classical work of Goodman--Pollack gives
\(\exp(O(N\log N))\)-type bounds for fixed-dimensional point configurations
\cite{GoodmanPollack}.  Alon's refinement gives the fixed-dimensional
asymptotic order for configurations: for fixed \(d\), the number of labeled
order types of \(N\) points in \(\R^d\) has the form
\[
    N^{d^2N+o(N)}
\]
\cite{AlonMatroids}.  The same work gives \(N^{\Theta(N)}\) upper and lower
bounds for fixed-rank real matroids.  Thus affine order types in \(\R^3\) live
at scale \(N^{9N+o(N)}\), while rank-four vector matroid data live at
\(N^{\Theta(N)}\).  Our crossing discriminant uses only a structured subset of
the four-point determinant data, together with interval inequalities in
projection charts.  A proof organized through order types may therefore give a
sharper version of Theorem~\ref{thm:upper}, and it would explain directly why
the relevant growth is of the form \(N^{O(N)}\).

Gros--Ramirez Alfonsin's strong geometry is the closest known reconstruction
result to this viewpoint \cite{GrosRamirez}.  It augments the ordinary affine
chirotope by a wedge chirotope encoding relative positions of spanned
hyperplanes, and proves that isomorphic strong geometries yield isotopic
generic polygonal knots.  This is complementary to our argument: their data
classify isotopy type but are not accompanied by an asymptotic enumeration,
whereas our smaller wall family is used to count connected realization
chambers without claiming that its sign vector alone classifies a knot.

There is an important distinction between counting determinant sign data and
counting chambers.  Order-type and oriented-matroid bounds count possible sign
patterns of minors.  The chamber count asks for connected components of
realization sets for such sign patterns, after intersecting with the stick-knot
discriminant complement.  Realization spaces of oriented matroids can be
topologically complicated by universality phenomena \cite{Mnev,RichterGebert};
hence matroid enumeration alone does not automatically give the chamber count.
It does, however, give a natural intermediate invariant between the raw
semialgebraic arrangement and the knot type map.

A second route is genuinely fewnomial.  Existing fewnomial theorems control
positive solutions of sparse systems and the topology of sparse hypersurfaces
\cite{KhovanskiiFewnomials,BihanSottileFewnomial,BihanRojasSottile,
BihanSottileBetti}.  The chamber problem here asks for connected components of
sign conditions for \(O(N^2)\) sparse determinant walls on the Stiefel manifold.
The useful theorem would keep the dependence on the common sparse support and
the determinant incidence structure, rather than replacing the equations by an
arbitrary family of bounded-degree polynomials.

\begin{question}[Sparse determinant arrangement problem]
Can one prove a sign-condition component bound for the sparse determinant-wall
arrangement on \(\St_3(H)\) that improves the exponent or leading constant in
the \(N^{O(N)}\) bound of Theorem~\ref{thm:upper}?
\end{question}

There is also a lower-bound version of the same problem.  Fixing a regular
planar shadow, the over/under data are signs of height differences at its
crossings, hence are constrained by a finite system of linear inequalities in
the vertex heights.  This gives a concrete feasibility problem for crossing
assignments.  If one can construct shadows for which many feasible assignments
have distinct knot invariants, then the fewnomial/determinant chamber structure
would also become a knot production mechanism, not only an upper-bound
mechanism.

\section*{Acknowledgments}

The authors used OpenAI's GPT-5.6-sol as a research and coding assistant to
help strengthen the results, improve the exposition, and generate and audit
the Python and Lean~4 code accompanying this paper.  The authors reviewed the
resulting mathematics and code and take responsibility for the final content.


\begin{thebibliography}{99}

\bibitem{AdamsStick}
C. Adams, B. M. Brennan, D. L. Greilsheimer, and A. K. Woo,
``Stick numbers and composition of knots and links,''
\emph{Journal of Knot Theory and Its Ramifications} 6 (1997), 149--161.
\url{https://doi.org/10.1142/S0218216597000121}.

\bibitem{BochnakCosteRoy}
J. Bochnak, M. Coste, and M.-F. Roy,
\emph{Real Algebraic Geometry},
Ergebnisse der Mathematik und ihrer Grenzgebiete, vol. 36,
Springer, 1998.

\bibitem{AlonMatroids}
N. Alon,
``The number of polytopes, configurations and real matroids,''
\emph{Mathematika} 33 (1986), 62--71.
\url{https://web.math.princeton.edu/~nalon/PDFS/Publications2/The%20number%20of%20polytopes%2C%20configurations%2C%20and%20real%20matroids.pdf}.

\bibitem{BaroneBasu}
S. Barone and S. Basu,
``Refined bounds on the number of connected components of sign conditions on a variety,''
\emph{Discrete \& Computational Geometry} 47 (2012), 577--597.
\url{https://arxiv.org/abs/1104.0636}.

\bibitem{BihanRojasSottile}
F. Bihan, J. M. Rojas, and F. Sottile,
``On the sharpness of fewnomial bounds and the number of components of a fewnomial hypersurface,''
\url{https://arxiv.org/abs/math/0701667}.

\bibitem{BihanSottileFewnomial}
F. Bihan and F. Sottile,
``New fewnomial upper bounds from Gale dual polynomial systems,''
\emph{Moscow Mathematical Journal} 7 (2007), 387--407.
\url{https://arxiv.org/abs/math/0609544}.

\bibitem{BihanSottileBetti}
F. Bihan and F. Sottile,
``Betti number bounds for fewnomial hypersurfaces via stratified Morse theory,''
\emph{Proceedings of the American Mathematical Society} 137 (2009), 2825--2833.
\url{https://arxiv.org/abs/0801.2554}.

\bibitem{Calvo}
J. A. Calvo,
``Geometric knot spaces and polygonal isotopy,''
\emph{Journal of Knot Theory and Its Ramifications} 10 (2001), 245--267.
\url{https://arxiv.org/abs/math/9904037}.

\bibitem{CantarellaEtAl2025}
J. Cantarella, A. Rechnitzer, H. Schumacher, and C. Shonkwiler,
New upper bounds for stick numbers,
preprint, arXiv:2508.18263, 2025.
\url{https://arxiv.org/abs/2508.18263}.

\bibitem{EdwardsKirby}
R. D. Edwards and R. C. Kirby,
``Deformations of spaces of imbeddings,''
\emph{Annals of Mathematics} 93 (1971), 63--88.

\bibitem{ElrifaiMorton}
E. A. Elrifai and H. R. Morton,
``Algorithms for positive braids,''
\emph{Quarterly Journal of Mathematics} 45 (1994), 479--497.
\url{https://doi.org/10.1093/qmath/45.4.479}.

\bibitem{FreedmanHeWang}
M. H. Freedman, Z.-X. He, and Z. Wang,
\emph{M\"obius energy of knots and unknots},
\emph{Annals of Mathematics} 139 (1994), 1--50.
\url{https://doi.org/10.2307/2946626}.

\bibitem{GoodmanPollack}
J. E. Goodman and R. Pollack,
``Upper bounds for configurations and polytopes in \(\R^d\),''
\emph{Discrete \& Computational Geometry} 1 (1986), 219--228.
\url{https://eudml.org/doc/186385}.

\bibitem{GrosRamirez}
B. Gros and J. L. Ramirez Alfonsin,
\emph{Strong geometry: knots},
preprint, arXiv:2504.00197, version 3, 2025.
\url{https://arxiv.org/abs/2504.00197}.

\bibitem{HuhOh}
Y. Huh and S. Oh,
``An upper bound on stick numbers of knots,''
\emph{Journal of Knot Theory and Its Ramifications} 20 (2011), 741--747.
\url{https://doi.org/10.1142/S0218216511008966}.

\bibitem{KhovanskiiFewnomials}
A. G. Khovanskii,
\emph{Fewnomials},
Translations of Mathematical Monographs, vol. 88,
American Mathematical Society, Providence, RI, 1991.

\bibitem{KnotInfoStick}
KnotInfo,
``Stick number.''
\url{https://knotinfo.org/descriptions/polygon_index.html}.

\bibitem{MalyutinStupakovArc}
A. Malyutin and M. Stupakov,
``On the number of knots with a given arc index,''
PDMI preprint 07/2022.
\url{https://www.pdmi.ras.ru/preprint/2022/22-07.html}; full text:
\url{http://ftp.pdmi.ras.ru/pub/publicat/preprint/2022/07-22.pdf.gz}.

\bibitem{MillettPhysical}
K. C. Millett,
``Physical knot theory: an introduction to the study of the influence of
knotting on the spatial characteristics of polymers,''
selected lectures presented at the Abdus Salam International Centre for
Theoretical Physics, 2009; manuscript dated 2010.
\url{https://web.math.ucsb.edu/~millett/Preprints/MillettTriesteLectures.pdf}.

\bibitem{Mnev}
N. E. Mn\"ev,
``The universality theorems on the classification problem of configuration
varieties and convex polytopes varieties,''
in \emph{Topology and Geometry: Rohlin Seminar},
Lecture Notes in Mathematics, vol. 1346, Springer, 1988, 527--543.

\bibitem{RandellSimonTokle}
R. Randell, J. Simon, and J. Tokle,
M\"obius transformations of polygons and partitions of \(3\)-space,
\emph{Journal of Knot Theory and Its Ramifications} 17 (2008), 1401--1413.
\url{https://doi.org/10.1142/S0218216508006671}.

\bibitem{PollackRoy}
R. Pollack and M.-F. Roy,
``On the number of cells defined by a set of polynomials,''
\emph{Comptes Rendus de l'Acad\'emie des Sciences, S\'erie I} 316 (1993),
573--577.
\url{https://mariefrancoiseroy.pages.math.cnrs.fr/MFRoymathpublications.html}.

\bibitem{RichterGebert}
J. Richter-Gebert,
\emph{Realization Spaces of Polytopes},
Lecture Notes in Mathematics, vol. 1643, Springer, 1996.
\url{https://link.springer.com/book/10.1007/BFb0093761}.

\bibitem{StanleyAlternating}
R. P. Stanley,
``A survey of alternating permutations,''
in \emph{Combinatorics and Graphs}, Contemporary Mathematics, vol. 531,
American Mathematical Society, 2010, 165--196.
\url{https://doi.org/10.1090/conm/531/10466}.

\end{thebibliography}
\end{document}